\documentclass[11pt]{article}

\usepackage[margin=1in]{geometry}
\usepackage{amsthm,amsmath,amssymb,latexsym}
\usepackage{graphicx,float}
\usepackage{tikz}
\usetikzlibrary{positioning,fit,backgrounds,calc,arrows.meta}
\usepackage{xcolor}
\usepackage[affil-it]{authblk}
\usepackage{indentfirst}
\usepackage[section]{placeins}
\usepackage{enumitem}
\usepackage{hyperref}

\hypersetup{
  colorlinks=true,
  linkcolor=blue,
  filecolor=blue,
  urlcolor=blue,
  citecolor=cyan
}

\newtheorem{theorem}{Theorem}
\newtheorem{lemma}{Lemma}
\newtheorem{claim}{Claim}

\newcommand{\N}{\mathbb{N}}

\newcommand{\dist}{\operatorname{dist}}
\newcommand{\e}{\mathrm{e}}
\newcommand{\dcol}[2]{d_{\mathrm{col},#1}(#2)}

\newcommand{\less}{\setminus}

\begin{document}

\title{New upper bound for the  Ramsey number of odd cycles}
\author{Ting HUANG, Jiabao YANG, Yaojun CHEN\footnote{Corresponding author. Email: yaojunc@nju.edu.cn}\\
 \small{School of Mathematics, Nanjing University, Nanjing 210093, P.R. CHINA}}
\date{}
\maketitle

\begin{quote}
\noindent {\bf Abstract:}
The \emph{$k$-color Ramsey number} $R_k(C_{2\ell+1})$ is the least integer $n$ such that any $k$-edge-coloring of a complete graph $K_n$ has a monochromatic odd cycle $C_{2\ell+1}$. 
Axenovich, Cames van Batenburg, Janzer, Michel, and Rundstr\"om~(JCT-B, 2026) recently proved
\[
R_k(C_{2\ell+1})\le (4\ell-2)^k k^{k/\ell}+1,
\]
and Miyazaki, Mulrenin, Pohoata, and Zheng further improved the factor $k^{k/\ell}$ to $(k!)^{1/\ell}$.
As Jenssen and Skokan (AM, 2021) determined $R_k(C_{2\ell+1})$ for fixed $k$ and sufficiently large $\ell$, it becomes even more interesting to seek better bound for fixed $\ell$ and sufficiently large $k$. In this paper, we show
\[
R_k(C_{2\ell+1})
\le
\frac{2\ell}{2\ell-1}(2\ell-1)^k(k!)^{1/\ell}
\exp\!\left(k^{1-1/\ell}+O_\ell\!\left(k^{1-2/\ell}+\log k\right)\right)+1
\]
for every fixed $\ell\ge 2$ and sufficiently large $k$, which improves the bound of Miyazaki et al. by a factor $2^{k-o(k)}$, and the bound of Axenovich et al. by a factor $(2\e^{1/\ell})^{k-o(k)}$.

 
\end{quote}

\section{Introduction}

The \emph{$k$-color Ramsey number} $R_k(H)$ of a graph $H$ is the smallest integer $n$ such that every $k$-edge-coloring of the complete graph $K_n$ contains a monochromatic copy of $H$. Let $C_n$ denote the cycle on $n$ vertices. Even in the case $H=C_3$, determining the asymptotic growth of $R_k(C_3)$ remains a major open problem. This is precisely the Schur--Erdős problem, which asks for the rate of growth of $R_k(C_3)$ as the number of colors $k$ tends to infinity. In 1917, Schur~\cite{S16} established the bounds
$$
\Omega(3^k) \le R_k(C_3) \le O(k!).
$$
The factorial upper bound, apart from refinements to the absolute constant, has remained asymptotically unchanged, while the lower bound has been improved to $\Omega(3.28^k)$ by Ageron, Casteras, Pellerin, Portella, Rimmel, and Tomasik~\cite{ACPPRT21}. Erd\H{o}s conjectured that the true growth rate is exponential in $k$, specifically $R_k(C_3)=2^{\Theta(k)}$~\cite{CG98}, and offered monetary prizes for this and several closely related problems.

For longer odd cycles, Bondy and Erd\H{o}s~\cite{BE73} and Erd\H{o}s and Graham~\cite{EG73} established the classical bounds 
$$\ell\cdot 2^k+1\le R_k(C_{2\ell+1})\le 2\ell\cdot (k+2)!.$$
The behavior depends strongly on which parameter is fixed.
  If $k$ is fixed and $\ell$ is sufficiently large, Jenssen and Skokan~\cite{JS21} proved the exact formula
$R_k(C_{2\ell+1})=\ell \cdot 2^k+1$, which implies the lower bound is sharp.
However, this fails for fixed $\ell$ and sufficiently large $k$. Day and Johnson~\cite{DJ17} proved that there exists a constant $\delta=\delta(\ell)>0$ such that
$
R_k(C_{2\ell+1}) \ge 2\ell \cdot (2+\delta)^{k-1}$.

The first substantial upper bounds in the latter regime were obtained by Li~\cite{Li09} and by Lin and Chen~\cite{LC19}. 
Li~\cite{Li09} proved that there is a constant $c$ such that $R_k(C_5)\le c^k\sqrt{k!}$.
Lin and Chen~\cite{LC19} later extended this result to all longer odd cycles, showing that for every $\ell\ge2$, $R_k(C_{2\ell+1})\le c^k\sqrt{k!}$, where the constant $c$ depends on $\ell$.
Li~\cite{Li09} further showed that, under the wide open additional assumption that every Ramsey graph for $R_k(C_{2\ell+1})$ is nearly regular (every Ramsey graph for $R_k(C_{2\ell+1})$ has minimum degree at least an $\varepsilon$-fraction of its average degree), one could improve the estimate to $R_k(C_{2\ell+1})\le c^k (k!)^{1/\ell}$. Moreover, Li~\cite{Li09} proposed a conjecture, asserting that $R_k(C_{2\ell+1})\le o((k!)^{1/\ell})$ as $k\to\infty$. 
This conjecture is motivated by Fox's conjecture~\cite{ACBJMR2026}, which states that for every $\varepsilon>0$, there exists an $\ell$ such that $R_k(C_{2\ell+1}) \le k^{\varepsilon k}$ for all sufficiently large $k$.

 Axenovich, Cames van Batenburg, Janzer, Michel, and Rundstr\"om~\cite{ACBJMR2026} settled Fox's conjecture and obtained the first improvement in the exponent, beyond an absolute constant factor, since the 1973 work of Bondy and Erd\H{o}s.

\begin{theorem}[Axenovich et al.~\cite{ACBJMR2026}]\label{thm:ACBJMR2026}
For all $k,\ell\in\N$,
\(
R_k(C_{2\ell+1})\le (4\ell-2)^k k^{k/\ell}+1.
\)
\end{theorem}

Miyazaki, Mulrenin, Pohoata, and Zheng~\cite{Miyazaki2026} then conditioned the weight increment on the current color degree, thereby obtained the following factorial refinement, which improves the bound in Theorem~\ref{thm:ACBJMR2026} by a multiplicative factor of roughly $\e^{k/\ell}$.

\begin{theorem}[Miyazaki et al.~\cite{Miyazaki2026}]\label{thm:Miyazaki2026}
For all $k,\ell\in\N$,
\(
R_k(C_{2\ell+1})\le (4\ell-2)^k (k!)^{1/\ell}+1.
\)
\end{theorem}

Although Jenssen and Skokan~\cite{JS21} determined the exact value of
$R_k(C_{2\ell+1})$ for fixed $k$ and sufficiently $\ell$, it is far more away from being known that the value of $R_k(C_{2\ell+1})$ for fixed $\ell$ and sufficiently large $k$. This motivates us to improve the upper bounds established in Theorems \ref{thm:ACBJMR2026} and \ref{thm:Miyazaki2026} along the direction.

\vskip 2mm
Before stating our main results, we need some additional notation and concepts.

Fix $\ell\ge1$.  For every integer $d\ge2$, define
\begin{equation}\label{eq:rho-intro}
\rho_d=\min\left\{\rho\in[1,\infty):1+\rho+\rho^2+\cdots+\rho^{\ell}\ge d\right\}.
\end{equation}
An edge-coloring is called \emph{$k$-local} if at most $k$ distinct colors are incident with each vertex; the total number of colors used in the graph may be larger than $k$.
Let $G$ be an edge-colored graph.  
A \emph{monochromatic graph} of $G$ means the spanning subgraph consisting of all edges of one fixed color.  
More precisely, for a color $c$, the \emph{color-$c$ graph} $G_c$ is the spanning subgraph of $G$ whose edge set consists exactly of the edges colored with $c$.

Our first main result is a local-coloring theorem.

\begin{theorem}\label{thm:main}
Let $k,\ell\in\N$.  If $K_n$ has a $k$-local edge-coloring in which every monochromatic graph is $C_{2\ell+1}$-free, then
\(
n\le
2\ell(2\ell-1)^{k-1}
\prod_{d=2}^{k}\bigl(1+\rho_d+\rho_d^{-1}\bigr).
\)

\end{theorem}
The empty product is interpreted as $1$.
Since every $k$-edge-coloring is also $k$-local, Theorem~\ref{thm:main} immediately yields the desired Ramsey bound.

\begin{theorem}\label{thm:ramsey}
For all $k,\ell\in\N$,
\[
R_k(C_{2\ell+1})
\le
\left\lfloor
2\ell(2\ell-1)^{k-1}
\prod_{d=2}^{k}\bigl(1+\rho_d+\rho_d^{-1}\bigr)
\right\rfloor+1.
\]
\end{theorem}

For fixed $\ell\ge2$ and $k$ sufficiently large, the product has a simple asymptotic form.

\begin{theorem}\label{thm:asymptotic-intro}
For every fixed integer $\ell\ge2$ and $k$ sufficiently large,
\[
R_k(C_{2\ell+1})
\le
\frac{2\ell}{2\ell-1}(2\ell-1)^k(k!)^{1/\ell}
\exp\!\left(k^{1-1/\ell}+O_\ell\!\left(k^{1-2/\ell}+\log k\right)\right)+1.
\]
\end{theorem}

Theorem~\ref{thm:asymptotic-intro} shows that our bound is smaller than the bound in Theorem~\ref{thm:Miyazaki2026} by a factor $2^{k-o(k)}$.  Compared with Theorem~\ref{thm:ACBJMR2026}, it gains both this factor and the factorial improvement from $k^{k/\ell}$ to $(k!)^{1/\ell}$; altogether the improvement factor is $(2\e^{1/\ell})^{k-o(k)}$.

The proof has two new ingredients.
The first is an \emph{exact-layer deletion}.  The argument of Axenovich et al. controls an entire monochromatic ball of radius $\ell$.  Because the even and odd BFS layers require separate palettes, this costs $4\ell-2$ vertex colors.  We keep vertices from only one exact layer.  A theorem of Erd\H{o}s, Faudree, Rousseau, and Schelp then costs only $2\ell-1$ colors, and the two adjacent layers form a small interface containing every color-$c$ edge incident with the retained layer.

The second ingredient is a sharper, color-degree-sensitive potential.  Suppose the current minimum color degree is $d$.  We choose a vertex $v$ of color degree $d$ and an incident color $c$ with a large weighted first layer.  If the first layer is cheap, we delete it and make $v$ lose color $c$.  Otherwise, the weights of the exact color-$c$ layers cannot grow too quickly at every step: if they did, the first $\ell+1$ layers would exceed the total budget available from the $d$ colors at $v$.  The threshold $\rho_d$ in~\eqref{eq:rho-intro} is chosen precisely so that
\(
1+\rho_d+\cdots+\rho_d^\ell\ge d
\)
turns this observation into a contradiction.
Thus some layer $L_i$ has a relatively light three-layer neighborhood $L_{i-1}\cup L_i\cup L_{i+1}$.
Inside $L_i$, the layer-coloring theorem provides a color-$c$ independent set $I_i$ containing at least a $1/(2\ell-1)$ fraction of the layer weight.  We keep $I_i$ and delete the rest of the three-layer interface.  Every vertex of $I_i$ then loses color $c$.
The multiplier
\(
(2\ell-1)\bigl(1+\rho_d+\rho_d^{-1}\bigr)
\)
is exactly large enough to balance the loss:
\[
\underbrace{(2\ell-1)\bigl(1+\rho_d+\rho_d^{-1}\bigr)w(I_i)}_{\text{new weight of the retained core}}
\ge
\underbrace{w(L_{i-1})+w(L_i)+w(L_{i+1})}_{\text{weight removed from the interface}}.
\]
This makes the total weight nondecreasing under each deletion, while the graph becomes smaller.  Induction then forces the original total weight to be at most one.

The paper is organized as follows. Section~\ref{preliminaries} introduces notation, the exact-layer coloring theorem, and the parameters used in the potential. Section~\ref{key result} states a key weighted estimate; we first use it to derive Theorems~\ref{thm:main} and~\ref{thm:ramsey} in Subsection~\ref{thm3and4}, and then prove the estimate in Subsection~\ref{thm7}. Section~\ref{thm5} gives the proof of Theorem~\ref{thm:asymptotic-intro}, which derives an asymptotic bound from our upper bound. Section~\ref{comparison} compares our bound with two recent ones.

\section{Preliminaries}\label{preliminaries}

\subsection{Notation and elementary facts}

For a positive integer $n$, write $[n]=\{1,2,\ldots,n\}$.
Given a graph $G$, let $\chi(G)$ and $\dist_G(x,y)$ be its chromatic number and the distance between two vertices $x,y$ in the same component.

Let $G$ be an edge-colored graph.  For a vertex $v\in V(G)$, define
\[
\mathcal C_G(v)=\{c:\text{an edge of color $c$ is incident with $v$}\}
\]
and
\(
\dcol{G}{v}=|\mathcal C_G(v)|.
\)
We call $\dcol{G}{v}$ the \emph{color degree} of $v$.  When the ambient graph is clear, we simply write $d_{\mathrm{col}}(v)$.

For a color $c$, a vertex $v$, and an integer $i\ge0$, define the exact color-$c$ layer
\[
N_c^i(v)=\{x\in V(G):\dist_{G_c}(v,x)=i\}.
\]
We use the convention $N_c^0(v)=\{v\}$.  Vertices outside the color-$c$ component of $v$ belong to none of these layers.  For an uncolored graph $H$, we similarly write $N_H^i(v)$, or simply $N^i(v)$ when $H$ is clear.

If $G'$ is a subgraph of $G$ containing $v$, we say that \emph{$v$ loses color $c$} when passing from $G$ to $G'$ if
$c\in\mathcal C_G(v)\less\mathcal C_{G'}(v)$.
Deleting vertices can only make a surviving vertex lose colors; it can never create a new incident color.

We first record the basic interface property of BFS layers.

\begin{lemma}\label{lem:adjacent-layers}
Let $x\in N_c^i(v)$ and let $xy$ be an edge of color $c$.  Then
\(
y\in N_c^{i-1}(v)\cup N_c^i(v)\cup N_c^{i+1}(v),
\)
where $N_c^{-1}(v)=\varnothing$.
\end{lemma}

\begin{proof}
A shortest color-$c$ path from $v$ to $x$, followed by the edge $xy$, gives
$\dist_{G_c}(v,y)\le i+1$.
Interchanging $x$ and $y$ gives
$i=\dist_{G_c}(v,x)\le\dist_{G_c}(v,y)+1$.
Hence $i-1\le\dist_{G_c}(v,y)\le i+1$, and the distance is an integer.
\end{proof}

The next theorem is the structural input concerning odd cycles.

\begin{theorem}[Erd\H{o}s, Faudree, Rousseau, and Schelp~\cite{EFRS78}]\label{thm:layer-color}
Let $H$ be a graph containing no cycle of length $2\ell+1$.  Then, for every vertex $v\in V(H)$ and every $1\le i\le\ell$,
\(
\chi\bigl(H[N_H^i(v)]\bigr)\le 2\ell-1.
\)
\end{theorem}

The point for us is that a \emph{single exact layer} costs only $2\ell-1$ colors.  By contrast, coloring the union of the first $\ell$ layers by giving the even and odd layers separate palettes costs $2(2\ell-1)=4\ell-2$, which is the factor appearing in Theorems~\ref{thm:ACBJMR2026} and~\ref{thm:Miyazaki2026}.
The proof of Theorem~\ref{thm:layer-color} can be viewed as follows: an ordered BFS tree converts a sufficiently long monotone path inside one layer into a $C_{2\ell+1}$, and the maximum length of a monotone path starting at a vertex then gives a $(2\ell-1)$-coloring.

\subsection{The parameters and the weight multipliers}

For $d\ge2$, let $\rho_d$ be defined by~\eqref{eq:rho-intro}, and set
\begin{equation}\label{eq:B-def}
B_d=1+\rho_d+\rho_d^{-1}.
\end{equation}
Define
\begin{equation}\label{eq:A-def}
A_1=2\ell,
\quad
A_d=(2\ell-1)B_d=(2\ell-1)\bigl(1+\rho_d+\rho_d^{-1}\bigr)
\quad \text{for } d\ge2,
\end{equation}
and
\begin{equation}\label{eq:P-def}
P_0=1,
\quad
P_t=\prod_{j=1}^{t}A_j
\quad \text{for } t\ge1.
\end{equation}

We collect the elementary properties needed later.

\begin{lemma}\label{lem:rho-properties}
For every integer $d\ge2$, the number $\rho_d$ exists.  Moreover:
\begin{enumerate}[label=\textup{(\roman*)}]
\item if $2\le d\le\ell+1$, then $\rho_d=1$;
\item if $d>\ell+1$, then $\rho_d>1$ is the unique solution of
$1+\rho+\cdots+\rho^\ell=d$;
\item the sequence $(\rho_d)_{d\ge2}$ is nondecreasing.
\end{enumerate}
\end{lemma}

\begin{proof}
For $\rho\ge1$, let $f(\rho)=\sum_{j=0}^{\ell}\rho^j$.
The function $f$ is continuous and strictly increasing, with
$f(1)=\ell+1$ and $f(\rho)\to\infty$ as $\rho\to\infty$.
This proves (i) and (ii).  If $d_1\le d_2$, then every $\rho$ satisfying
$f(\rho)\ge d_2$ also satisfies $f(\rho)\ge d_1$; hence the least admissible value for $d_1$ is no larger than the least admissible value for $d_2$, proving (iii).
\end{proof}

\begin{lemma}\label{lem:A-monotone}
The sequence $A_1,A_2,\ldots$ is nondecreasing, and every $A_d>1$.
\end{lemma}

\begin{proof}
The function $g(\rho)=1+\rho+\rho^{-1}$ is nondecreasing on $[1,\infty)$ because
$g'(\rho)=1-\rho^{-2}\ge0$.
Lemma~\ref{lem:rho-properties} therefore implies that $A_d=(2\ell-1)g(\rho_d)$ is nondecreasing for $d\ge2$.
Also $\rho_2=1$, so
$A_2=3(2\ell-1)\ge2\ell=A_1$.
Finally, $A_1=2\ell\ge2$.
\end{proof}

\begin{lemma}\label{lem:weight-gain}
Let $0\le s<t$.  Then
\(
\frac{P_t}{P_s}=\prod_{j=s+1}^{t}A_j\ge A_t.
\)
Consequently, if a vertex $x$ has color degree $t\ge d$ in $G$ and loses at least one color when passing to an induced subgraph $G'$, then
\[
\frac{P_{\dcol{G}{x}}}{P_{\dcol{G'}{x}}}\ge A_d.
\]
\end{lemma}

\begin{proof}
The product $P_t/P_s$ contains the factor $A_t$, and every other factor is larger than $1$.  If $t=\dcol{G}{x}\ge d$ and $s=\dcol{G'}{x}\le t-1$, then the first assertion and Lemma~\ref{lem:A-monotone} give
$P_t/P_s\ge A_t\ge A_d$.
\end{proof}

\section{A key weighted statement}\label{key result}

For an edge-colored graph $G$ and a vertex $x\in V(G)$, define
\begin{equation}\label{eq:vertex-weight}
w_G(x)=\frac{1}{P_{\dcol{G}{x}}}.
\end{equation}
For $X\subseteq V(G)$, let
\begin{equation}\label{eq:set-weight}
w_G(X)=\sum_{x\in X}w_G(x),
\end{equation}
and write
\begin{equation}\label{eq:total-weight}
W(G)=w_G(V(G)).
\end{equation}

The following weighted statement is the key to the proof.

\begin{theorem}\label{thm:weight-at-most-one}
Let $G$ be a complete graph with a $k$-local edge-coloring in which every color graph is $C_{2\ell+1}$-free.  Then
\(
W(G)\le1.
\)
\end{theorem}

\subsection{Proofs of Theorems~\ref{thm:main} and~\ref{thm:ramsey}}\label{thm3and4}
We first derive Theorems~\ref{thm:main} and~\ref{thm:ramsey} based on Theorem~\ref{thm:weight-at-most-one}, and then turn to the proof of Theorem~\ref{thm:weight-at-most-one}.

\begin{proof}[Proof of Theorem~\ref{thm:main}]
Since the coloring is $k$-local, $\dcol{G}{x}\le k$ for every vertex $x$.
The sequence $P_t$ is increasing, so
\(
w_G(x)=1/P_{\dcol{G}{x}}\ge1/P_k
\).
Summing over all $n$ vertices gives $W(G)\ge n/P_k$.
By Theorem~\ref{thm:weight-at-most-one}, $W(G)\le1$, and hence 
\[
n\le P_k=A_1\prod_{d=2}^{k}A_d
=2\ell(2\ell-1)^{k-1}
\prod_{d=2}^{k}\bigl(1+\rho_d+\rho_d^{-1}\bigr),
\]
which is the required bound.
\end{proof}

\begin{proof}[Proof of Theorem~\ref{thm:ramsey}]
Suppose that a $k$-edge-coloring of $K_n$ contains no monochromatic $C_{2\ell+1}$.
It is automatically $k$-local, so Theorem~\ref{thm:main} gives $n\le P_k$.
Therefore every $k$-edge-coloring of $K_{\lfloor P_k\rfloor+1}$ contains a monochromatic $C_{2\ell+1}$, and the stated inequality follows from the definition of the Ramsey number.
\end{proof}

\subsection{Proof of Theorem~\ref{thm:weight-at-most-one}} \label{thm7}

We use induction on $n=|V(G)|$.
If $n=1$, the unique vertex has color degree $0$, so its weight is $P_0^{-1}=1$.
Assume $n\ge2$ and that the theorem holds for every smaller complete graph satisfying the same hypotheses.

Choose a vertex $v$ of minimum color degree, and let
\(
d=\dcol{G}{v}.
\)
Since $G$ is complete and has at least two vertices, $d\ge1$.
We distinguish the cases $d=1$ and $d\ge2$.

\medskip
\noindent\textbf{Case 1: $d=1$.}
\medskip

Let $c$ be the unique color incident with $v$.
Every edge from $v$ has color $c$, so
$N_c^1(v)=V(G)\less\{v\}$.
The color graph $G_c$ is $C_{2\ell+1}$-free, and Theorem~\ref{thm:layer-color} with $i=1$ gives
\[
\chi\bigl(G_c[V(G)\less\{v\}]\bigr)\le2\ell-1.
\]
Color this set properly with at most $2\ell-1$ vertex colors, and assign $v$ one additional color.  We obtain a proper coloring of $G_c$ with at most
$2\ell=A_1$ vertex colors.

Let $S$ be a color class of maximum weight.  The classes partition $V(G)$, so
\begin{equation}\label{eq:d1-heavy-class}
w_G(S)\ge\frac{W(G)}{A_1}.
\end{equation}
The set $S$ is independent in $G_c$.  Since $n\ge2$ and $v$ is joined by a color-$c$ edge to every other vertex, $S\ne V(G)$.
Thus $G'=G[S]$ is a strictly smaller complete graph.
It remains $k$-local, and every one of its color graphs is an induced subgraph of the corresponding color graph of $G$, so it is still $C_{2\ell+1}$-free.

Every vertex of $S$ loses color $c$ in $G'$, because $S$ contains no color-$c$ edge.  Since every vertex of $G$ has color degree at least the minimum value $d=1$, Lemma~\ref{lem:weight-gain} gives
\(w_{G'}(x)\ge A_1w_G(x)\) for \(x\in S\).
Consequently, by~\eqref{eq:d1-heavy-class},
\[
W(G')=\sum_{x\in S}w_{G'}(x)
\ge A_1w_G(S)
\ge W(G).
\]
The induction hypothesis gives $W(G')\le1$, and therefore $W(G)\le1$.

\medskip
\noindent\textbf{Case 2: $d\ge2$.}
\medskip

The $d$ colors incident with $v$ partition $V(G)\less\{v\}$ according to the color of the edge from $v$.
Hence
\[
\sum_{c\in\mathcal C_G(v)}w_G(N_c^1(v))=W(G)-w_G(v).
\]
By averaging, there is a color $c\in\mathcal C_G(v)$ such that
\[
w_G(N_c^1(v))\ge\frac{W(G)-w_G(v)}{d}.
\]
For $i\ge0$, we write \(L_i=N_c^i(v)\) and \(\alpha_i=w_G(L_i)\).
Then $L_0=\{v\}$ and
\(
W(G)-\alpha_0\le d\alpha_1.
\)
Only vertices in the color-$c$ component of $v$ occur in the layers, so
\begin{equation}\label{eq:layer-sum-budget}
\sum_{i\ge1}\alpha_i
\le W(G)-\alpha_0
\le d\alpha_1.
\end{equation}
This is the global budget against which the local three-layer costs will be compared.

\begin{claim}\label{claim:layer-choice}
At least one of the following holds:
\begin{enumerate}[label=\textup{(\alph*)}]
\item $\alpha_1\le(A_d-1)\alpha_0$;
\item there is an $i\in[\ell]$ with $\alpha_i>0$ such that
\(
\alpha_{i-1}+\alpha_i+\alpha_{i+1}\le B_d\alpha_i.
\)
\end{enumerate}
\end{claim}

\begin{proof}
Suppose, for a contradiction, that
\begin{equation}\label{eq:root-expensive}
\alpha_1>(A_d-1)\alpha_0
\end{equation}
and that, for every $i\in[\ell]$ with $\alpha_i>0$,
\begin{equation}\label{eq:all-layers-expensive}
\alpha_{i-1}+\alpha_i+\alpha_{i+1}>B_d\alpha_i.
\end{equation}
First note that
$$A_d-1-\rho_d
=(2\ell-1)(1+\rho_d+\rho_d^{-1})-1-\rho_d=(2\ell-2)+(2\ell-2)\rho_d+(2\ell-1)\rho_d^{-1}>0.$$
Thus~\eqref{eq:root-expensive} implies
\begin{equation}\label{eq:first-ratio}
\frac{\alpha_1}{\alpha_0}>\rho_d,
\end{equation}
so in particular $\alpha_1>0$.

We claim inductively that, for every $1\le i\le\ell+1$,
\begin{equation}\label{eq:ratio-induction}
\alpha_i>0
\qquad\text{and}\qquad
\frac{\alpha_i}{\alpha_{i-1}}>\rho_d.
\end{equation}
The case $i=1$ is~\eqref{eq:first-ratio}.
Assume~\eqref{eq:ratio-induction} holds for some $1\le i\le\ell$.
It follows from \eqref{eq:all-layers-expensive} and $B_d=1+\rho_d+\rho_d^{-1}$ that
\(
\alpha_{i+1}>(\rho_d+\rho_d^{-1})\alpha_i-\alpha_{i-1}.
\)
After division by $\alpha_i>0$,
\[
\frac{\alpha_{i+1}}{\alpha_i}
>\rho_d+\rho_d^{-1}-\frac{\alpha_{i-1}}{\alpha_i}.
\]
The induction hypothesis gives
$\alpha_{i-1}/\alpha_i<\rho_d^{-1}$, and hence
$\alpha_{i+1}/\alpha_i>\rho_d$.
This proves~\eqref{eq:ratio-induction}.

It follows that
\(
\alpha_j>\rho_d^{j-1}\alpha_1\)
for \(1\le j\le\ell+1\).
Therefore, by the definition of $\rho_d$,
\[
\sum_{j=1}^{\ell+1}\alpha_j
>
\alpha_1\sum_{j=0}^{\ell}\rho_d^j
\ge d\alpha_1,
\]
contradicting the budget~\eqref{eq:layer-sum-budget}.
\end{proof}

We now perform the deletion supplied by Claim~\ref{claim:layer-choice}.

\medskip
\noindent\textbf{Subcase 2.1: $\alpha_1\le(A_d-1)\alpha_0$.}
\medskip

Delete the first layer $L_1$ and let
\(
G'=G[V(G)\less L_1].
\)
The set $L_1$ is precisely the set of color-$c$ neighbors of $v$, so $v$ loses color $c$.  Since $d\ge2$, the neighbors of $v$ in its other incident colors remain, and hence
$\dcol{G'}{v}=d-1$.
Thus
\(
w_{G'}(v)=A_dw_G(v).
\)
Every other surviving vertex can only lose colors, so its weight does not decrease.  Consequently,
\begin{align*}
W(G')
\ge W(G)-\alpha_1-\alpha_0+A_d\alpha_0
=W(G)-\alpha_1+(A_d-1)\alpha_0
\ge W(G).
\end{align*}
The layer $L_1$ is nonempty because $c$ is incident with $v$, so $G'$ is strictly smaller than $G$.  The induction hypothesis gives $W(G')\le1$, and hence $W(G)\le1$.

\medskip
\noindent\textbf{Subcase 2.2: there is an $i\in[\ell]$ with $\alpha_i>0$ and
$\alpha_{i-1}+\alpha_i+\alpha_{i+1}\le B_d\alpha_i$.}
\medskip

The color graph $G_c$ is $C_{2\ell+1}$-free, so Theorem~\ref{thm:layer-color} gives
\(
\chi(G_c[L_i])\le2\ell-1.
\)
Partition $L_i$ into at most $2\ell-1$ color-$c$ independent sets, adding empty classes if necessary, and choose a class $I_i$ of maximum weight.  Then
\begin{equation}\label{eq:I-heavy}
w_G(I_i)\ge\frac{\alpha_i}{2\ell-1}.
\end{equation}

Delete
\(
D_i=L_{i-1}\cup(L_i\less I_i)\cup L_{i+1}
\)
and let
\(
G'=G[V(G)\less D_i].
\)
By Lemma~\ref{lem:adjacent-layers}, every color-$c$ neighbor of a vertex $x\in I_i$ lies in
$L_{i-1}\cup L_i\cup L_{i+1}$.  $L_{i-1}$ and $L_{i+1}$ are deleted, while a color-$c$ neighbor in $L_i$ cannot belong to $I_i$ because $I_i$ is independent in $G_c$.  Thus every vertex of $I_i$ loses color $c$.

\begin{figure}[htbp]
\centering
\begin{tikzpicture}[
  layer/.style={draw, rounded corners, minimum width=3.2cm, minimum height=1.35cm, align=center},
  keep/.style={draw, rounded corners, minimum width=2.25cm, minimum height=0.44cm, align=center},
  note/.style={font=\small}
]
\node[layer] (left) {$L_{i-1}$\\[1mm]{\small delete}};
\node[layer, right=0.7cm of left] (middle) {};
\node[layer, right=0.7cm of middle] (right) {$L_{i+1}$\\[1mm]{\small delete}};
\node[note, anchor=north] at ($(middle.north)+(0,0.10)$) {$L_i$};
\node[keep] (core) at ($(middle.center)+(0,0.02)$) {keep $I_i$};
\node[note] at ($(middle.center)+(0,-0.48)$) {delete $L_i\setminus I_i$};
\draw ($(left.east)+(0,0)$) -- ($(middle.west)+(0,0)$);
\draw ($(middle.east)+(0,0)$) -- ($(right.west)+(0,0)$);
\node[note, below=0.35cm of middle] {color-$c$ edges meet only the same or an adjacent layer};
\end{tikzpicture}
\caption{The exact-layer deletion.  
}
\label{fig:layer-deletion}
\end{figure}

The vertex $v$ was chosen with minimum color degree, so every $x\in I_i$ has
$\dcol{G}{x}\ge d$.  Lemma~\ref{lem:weight-gain} therefore gives \(w_{G'}(x)\ge A_dw_G(x)\) for \(x\in I_i\).
Every other surviving vertex has nondecreasing weight.  Hence, using~\eqref{eq:I-heavy} and $A_d=(2\ell-1)B_d$,
\begin{align*}
W(G')
\ge W(G)-(\alpha_{i-1}+\alpha_i+\alpha_{i+1})+A_dw_G(I_i)
\ge W(G)-B_d\alpha_i+A_d\frac{\alpha_i}{2\ell-1}
=W(G).
\end{align*}
Because $\alpha_i>0$, the layer $L_i$ is nonempty.  Every vertex of $L_i$ has a predecessor in $L_{i-1}$ on a shortest color-$c$ path from $v$, so $L_{i-1}$ is nonempty and is deleted.  Thus $G'$ is strictly smaller than $G$.  The induction hypothesis gives $W(G')\le1$, and therefore $W(G)\le1$.

Both cases are complete.

\section{Proof of Theorem~\ref{thm:asymptotic-intro}} \label{thm5}
 Throughout this section $\ell\ge2$ is fixed and $d\to\infty$.
For $d>\ell+1$, Lemma~\ref{lem:rho-properties} shows that $\rho_d$ is the unique solution of
\begin{equation}\label{eq:rho-equation}
1+\rho_d+\rho_d^2+\cdots+\rho_d^\ell=d.
\end{equation}

\begin{lemma}\label{lem:rho-expansion}
Fix $\ell \ge 2$ and let $x=d^{1/\ell}$.  Then
\[
\rho_d
=x-\frac1\ell-\frac{\ell+1}{2\ell^2}x^{-1}+O_\ell(x^{-2}).
\]
\end{lemma}

\begin{proof}
Set
\[
\widetilde\rho=x-\frac1\ell-\frac{\ell+1}{2\ell^2}x^{-1}
\qquad\text{and}\qquad
F(t)=\sum_{j=0}^{\ell}t^j.
\]
Then $F(\rho_d)=d$.
Write $c_0=-1/\ell$ and $c_1=-(\ell+1)/(2\ell^2)$, so
$\widetilde\rho=x+c_0+c_1x^{-1}$.
By the binomial theorem, we have
\begin{align*}
\widetilde\rho^\ell
&=x^\ell+\ell c_0x^{\ell-1}
 +\left(\ell c_1+\binom{\ell}{2}c_0^2\right)x^{\ell-2}
 +O_\ell(x^{\ell-3}),\\
\widetilde\rho^{\ell-1}
&=x^{\ell-1}+(\ell-1)c_0x^{\ell-2}+O_\ell(x^{\ell-3}),\\
\widetilde\rho^{\ell-2}
&=x^{\ell-2}+O_\ell(x^{\ell-3}).
\end{align*}
All lower powers together contribute $O_\ell(x^{\ell-3})$.  Hence
\begin{align*}
F(\widetilde\rho)
=x^\ell+(\ell c_0+1)x^{\ell-1} +
\left(\ell c_1+\binom{\ell}{2}c_0^2+(\ell-1)c_0+1\right)x^{\ell-2}
+O_\ell(x^{\ell-3}).
\end{align*}
The choices of $c_0$ and $c_1$ make the coefficients of $x^{\ell-1}$ and $x^{\ell-2}$ equal to zero, so
\[
F(\widetilde\rho)=x^\ell+O_\ell(x^{\ell-3})
=d+O_\ell(x^{\ell-3})=F(\rho_d)+O_\ell(x^{\ell-3}).
\]

From~\eqref{eq:rho-equation} and $\rho_d\ge1$,
\(
\rho_d^\ell\le d\le(\ell+1)\rho_d^\ell,
\)
so
\(
\frac{x}{(\ell+1)^{1/\ell}}\le\rho_d\le x.
\)
In particular, for all sufficiently large $x$, both $\rho_d$ and $\widetilde\rho$ lie in $[x/2,2x]$.
On this interval,
$F'(t)=\sum_{j=1}^{\ell}jt^{j-1}=\Theta_\ell(x^{\ell-1})$.
Applying the mean value theorem 
, we have
\[
|\rho_d-\widetilde\rho|
=
\frac{|F(\rho_d)-F(\widetilde\rho)|}{F'(\xi)}
=
\frac{O_\ell(x^{\ell-3})}{\Theta_\ell(x^{\ell-1})}
=O_\ell(x^{-2})
\]
for some $\xi$ between $\rho_d$ and $\widetilde\rho$.
\end{proof}

\begin{lemma}\label{lem:B-expansion}
For fixed $\ell\ge2$, $$B_d
=d^{1/\ell}+\frac{\ell-1}{\ell}+O_\ell(d^{-1/\ell})~\text{and}~ \log B_d
=\frac1\ell\log d+\frac{\ell-1}{\ell}d^{-1/\ell}
+O_\ell(d^{-2/\ell}).$$
\end{lemma}

\begin{proof}
Let $x=d^{1/\ell}$.  By Lemma~\ref{lem:rho-expansion},
$\rho_d=x-\frac{1}{\ell}+O_\ell(x^{-1})=x\left(1-\frac{1}{\ell x}+O_\ell(x^{-2})\right)$.
Applying the Taylor expansion $1/(1-u)=1+u+O(u^2)$, we have
$$\rho_d^{-1}=\frac{1}{x}\cdot \frac{1}{1-\frac{1}{\ell x}+O_\ell(x^{-2})}=
 \frac{1}{x}+\frac{1}{\ell x^2}+O_\ell(x^{-3}).$$
Therefore
$$B_d=1+\rho_d+\rho_d^{-1}=1+x-\frac1\ell+O_\ell(x^{-1})=x+\frac{\ell-1}{\ell}+O_\ell(x^{-1})=d^{1/\ell}+\frac{\ell-1}{\ell}+O_\ell(d^{-1/\ell}).$$
Note that
$B_d=x\left(1+\frac{\ell-1}{\ell}x^{-1}+O_\ell(x^{-2})\right).$
For $u=O(x^{-1})$, the Taylor expansion $\log(1+u)=u+O(u^2)$ yields
\begin{align*}
\log B_d=\log x+\frac{\ell-1}{\ell}x^{-1}+O_\ell(x^{-2})=\frac1\ell\log d+\frac{\ell-1}{\ell}d^{-1/\ell}+O_\ell(d^{-2/\ell}),
\end{align*}
as expected.
\end{proof}

\begin{lemma}\label{lem:power-sum}
For every fixed $0<\alpha<1$,
\(
\sum_{d=2}^{k}d^{-\alpha}
=\frac{k^{1-\alpha}}{1-\alpha}+O_\alpha(1).
\)
And
\(
\sum_{d=2}^{k}\frac1d=\log k+O(1).
\)
\end{lemma}

\begin{proof}
Both estimates follow from the integral comparison for the decreasing function $x^{-\alpha}$; the second is the case $\alpha=1$.
Indeed, since $x^{-\alpha}$ decreases, we have
$$\int_1^k x^{-\alpha}\,dx\le 1+\sum_{d=2}^{k}d^{-\alpha}\le 1+\int_1^k x^{-\alpha}\,dx \quad \text{and}\quad \int_1^k x^{-\alpha}\,dx=
\begin{cases}
\frac{k^{1-\alpha}-1}{1-\alpha},&0<\alpha<1,\\
\log k,&\alpha=1.
\end{cases}
$$
It follows that the expected inequalities hold.
\end{proof}

\begin{proof}[Proof of Theorem~\ref{thm:asymptotic-intro}]

By \eqref{eq:B-def},
$B_d=1+\rho_d+\rho_d^{-1}$.
It suffices to prove \begin{equation}\label{eq:product-asymptotic-intro}
\prod_{d=2}^{k}\bigl(1+\rho_d+\rho_d^{-1}\bigr)=\prod_{d=2}^{k}B_d
=(k!)^{1/\ell}
\exp\!\left(k^{1-1/\ell}+O_\ell\!\left(k^{1-2/\ell}+\log k\right)\right).
\end{equation}
Taking logarithms and applying Lemma~\ref{lem:B-expansion},
\begin{align}\label{eq:sum-logB}
\log\prod_{d=2}^{k}B_d=\sum_{d=2}^{k}\log B_d
=\frac1\ell\sum_{d=2}^{k}\log d
+\frac{\ell-1}{\ell}\sum_{d=2}^{k}d^{-1/\ell}
+O\left(\sum_{d=2}^{k}d^{-2/\ell}\right).
\end{align}

The first sum is $\log(k!)$.  By Lemma~\ref{lem:power-sum},
\(
\frac{\ell-1}{\ell}\sum_{d=2}^{k}d^{-1/\ell}
=k^{1-1/\ell}+O_\ell(1).
\)
For the error term,
\[
\sum_{d=2}^{k}d^{-2/\ell}
=
\begin{cases}
O(\log k),&\ell=2,\\
O_\ell(k^{1-2/\ell}),&\ell>2.
\end{cases}
\]
Substituting these estimates into~\eqref{eq:sum-logB} gives
\[
\log\prod_{d=2}^{k}B_d
=\frac1\ell\log(k!)+k^{1-1/\ell}
+O_\ell\bigl(k^{1-2/\ell}+\log k\bigr).
\]
Exponentiation proves~\eqref{eq:product-asymptotic-intro}.  Finally,
\(
2\ell(2\ell-1)^{k-1}
=\frac{2\ell}{2\ell-1}(2\ell-1)^k,
\)
which gives the stated Ramsey bound
\[
R_k(C_{2\ell+1})
\le
\frac{2\ell}{2\ell-1}(2\ell-1)^k(k!)^{1/\ell}
\exp\!\left(k^{1-1/\ell}+O_\ell\!\left(k^{1-2/\ell}+\log k\right)\right)+1,
\]
as required.
\end{proof}

\section{Comparison with the recent bounds}\label{comparison}

For fixed $\ell\ge2$, ignore the final additive $1$ and write $U_{\rm new}$ for the main term in Theorem~\ref{thm:asymptotic-intro},
$U_{\rm M}=(4\ell-2)^k(k!)^{1/\ell}$ for the bound in Theorem~\ref{thm:Miyazaki2026}, and
$U_{\rm A}=(4\ell-2)^k k^{k/\ell}$ for the bound in Theorem~\ref{thm:ACBJMR2026}.
Since $4\ell-2=2(2\ell-1)$, Theorem~\ref{thm:asymptotic-intro} gives
\[
\log\frac{U_{\rm new}}{U_{\rm M}}
=-k\log2+k^{1-1/\ell}
+O_\ell\bigl(k^{1-2/\ell}+\log k\bigr).
\]
Because $k^{1-1/\ell}=o(k)$, this ratio is $2^{-k+o(k)}$.
Thus the exact-layer step removes, up to a subexponential correction, one factor $2$ for each color.

Using Stirling's formula,
\(
\log\frac{(k!)^{1/\ell}}{k^{k/\ell}}
=-\frac{k}{\ell}+O_\ell(\log k).
\)
Therefore
\[
\log\frac{U_{\rm new}}{U_{\rm A}}
=-\left(\log2+\frac1\ell\right)k
+k^{1-1/\ell}
+O_\ell\bigl(k^{1-2/\ell}+\log k\bigr).
\]
Equivalently, the improvement factor over Theorem~\ref{thm:ACBJMR2026} is
$(2\e^{1/\ell})^{k-o(k)}$.
The two gains have different sources: the factor $2^{k-o(k)}$ comes from using one exact layer rather than an entire ball, while the factor $\e^{k/\ell-o(k)}$ comes from adapting the weight to the current color degree rather than using the worst-case value $k$ at every step.

The leading term in the logarithm of our upper bound is still
$\frac1\ell k\log k$.  Within the present single-layer framework this is natural: the geometric constraint
$1+\rho_d+\cdots+\rho_d^\ell\ge d$ forces $\rho_d=\Theta(d^{1/\ell})$, and multiplying this scale over $d=1,\ldots,k$ produces $(k!)^{1/\ell}$.  Any improvement of the leading coefficient would therefore require additional structure beyond a single retained BFS layer.

\section*{Acknowledgement}
\noindent This research is supported by the National Key R\&D Program of China under grant number 2024YFA1013900 and the NSFC under grant number 12471327.

\section*{Declaration}
\noindent\textbf{Conflict of interest.}
The authors declare that they have no known competing financial interests or personal relationships that could have appeared to influence the work reported in this paper.

\section*{Data availability}
No data were used for the research described in this paper.

\end{document}